\documentclass[12pt,twoside]{amsart}
\usepackage{amsmath, amsthm, amscd, amsfonts, amssymb, graphicx}
\usepackage{enumerate}
\usepackage[colorlinks=true,
linkcolor=blue,
urlcolor=cyan,
citecolor=red]{hyperref}
\usepackage{mathrsfs}
\newtheorem{theorem}{Theorem}[section]
\newtheorem{lemma}{Lemma}[section]
\newtheorem{remark}{Remark}[section]

\newtheorem{corollary}{Corollary}[section]
\newtheorem{example}{Example}[section]
\newtheorem{proposition}{Proposition}[section]
\numberwithin{equation}{section}

\begin{document}
	
\title{Further inequalities for the numerical radius of Hilbert space operators}
\author{Sara Tafazoli$^1$, Hamid Reza Moradi$^2$, SHIGERU FURUICHI$^3$ and PANACKAL HARIKRISHNAN$^4$}
\subjclass[2010]{Primary 47A12, Secondary 47A30, 15A60,  47A63.}
\keywords{Operator inequality, norm inequality, numerical radius, convex function, $f$-connection, weighted arithmetic-geometric mean inequality.}
\maketitle

\begin{abstract}
In this article, we present some new inequalities for numerical radius of Hilbert space operators via convex functions. Our results generalize and improve earlier results by El-Haddad and Kittaneh.  Among several results, we show that	if $A\in \mathbb{B}\left( \mathcal{H} \right)$ and $r\ge 2$, then
\[{{w}^{r}}\left( A \right)\le {{\left\| A \right\|}^{r}}-\underset{\left\| x \right\|=1}{\mathop{\inf }}\,{{\left\| {{\left| \left| A \right|-w\left( A \right) \right|}^{\frac{r}{2}}}x \right\|}^{2}}\]
where $w\left( \cdot \right)$ and $\left\| \cdot \right\|$ denote the numerical radius and usual operator norm, respectively.
\end{abstract}
\pagestyle{myheadings}
\markboth{\centerline {S. Tafazoli, H. R. Moradi, S. Furuichi \& P. K. Harikrishnan}}
{\centerline {Further inequalities for the numerical radius of Hilbert space operators}}
\bigskip
\bigskip
\section{Introduction}
Let $\mathbb{B}(\mathcal{H})$ denote the $C^*$-algebra of all bounded linear operators acting on a Hilbert space $\mathcal{H}.$ As
customary, we reserve $m$, $M$ for scalars. An operator $A$ on $\mathcal{H}$ is said to be positive (in symbol: $A\ge 0$) if $\left\langle Ax ,x  \right\rangle \ge 0$ for all $x \in \mathcal{H}$. We write $A>0$ if $A$ is positive and invertible. For self-adjoint operators $A$ and $B$, we write $A\ge B$ if $A-B$ is positive, i.e., $\left\langle Ax ,x  \right\rangle \ge \left\langle Bx ,x  \right\rangle $ for all $x \in \mathcal{H}$. We call it the usual order. In particular, for some scalars $m$ and $M$, we write $m\le A\le M$ if $m\left\langle x ,x  \right\rangle \le \left\langle Ax ,x  \right\rangle \le M\left\langle x ,x  \right\rangle $ for all $x \in \mathcal{H}$. Here $\left| A \right|={{\left( {{A}^{*}}A \right)}^{\frac{1}{2}}}$ is the absolute value of $A$.

 If $A\in\mathbb{B}(\mathcal{H})$, the usual operator norm and the numerical radius of $A$ are defined, respectively, by
$$\|A\|=\sup_{\|x\|=1}\|Ax\|\;\quad{\text{and}}\quad\;w(A)=\sup_{\|x\|=1}|\left<Ax,x\right>|.$$ 
The numerical radius satisfies 
\begin{equation}\label{14}
\frac{1}{2}\left\| A \right\|\le w \left( A \right)\le \left\| A \right\|,
\end{equation}
which show that $w \left( A \right)$ is a norm equivalent to $\left\| A \right\|$.   We also remark that if $R\left( A \right)\bot R\left( {{A}^{*}} \right)$, then $w \left( A \right)=\frac{1}{2}\left\| A \right\|$ (see, e.g.,  {\cite[Theorem 1.3.4]{8}}).

An improvement of the second inequality in \eqref{14} has been given in \cite[Theorem 1]{11}. It says that for $A\in \mathbb{B}\left( \mathcal{H} \right)$, 
\begin{equation}\label{15}
w\left( A \right)\le \frac{1}{2}\left\| \left| A \right|+\left| {{A}^{*}} \right| \right\|\le \frac{1}{2}\left( \left\| A \right\|+{{\left\| {{A}^{2}} \right\|}^{\frac{1}{2}}} \right).
\end{equation}
Consequently, if ${{A}^{2}}=0$, then $w \left( A \right)=\frac{\left\| A \right\|}{2}$. The first inequality of \eqref{15} was extended in \cite{5} in the following form:
\begin{equation}\label{28}
{{w}^{r}}\left( A \right)\le \frac{1}{2}\left\| {{\left| A \right|}^{2rv}}+{{\left| {{A}^{*}} \right|}^{2r\left( 1-v \right)}} \right\|,\quad\text{ }r\ge 1,~0<v<1.
\end{equation}
Also, in the same paper, it was shown that 
\begin{equation}\label{3}
{{\left\| A+B \right\|}^{2}}\le \left\| {{\left| A \right|}^{2}}+{{\left| B \right|}^{2}} \right\|+\left\| {{\left| {{A}^{*}} \right|}^{2}}+{{\left| {{B}^{*}} \right|}^{2}} \right\|.
\end{equation}
The following result concerning the product of two operators was proved in \cite{d}:
\begin{equation}\label{29}
{{w}^{r}}\left( {{B}^{*}}A \right)\le \frac{1}{2}\left\| {{\left| A \right|}^{2r}}+{{\left| B \right|}^{2r}} \right\|,\quad\text{ }r\ge 1.
\end{equation}
A general numerical radius inequality has been proved by Shebrawi and  Albadawi \cite{1}, it has been shown
that if $A,X,B\in \mathbb{B}\left( \mathcal{H} \right)$, then
\begin{equation}\label{20}
{{w}^{r}}\left( {{A}^{*}}XB \right)\le \frac{1}{2}\left\| {{\left( {{A}^{*}}{{\left| {{X}^{*}} \right|}^{2v}}A \right)}^{r}}+{{\left( {{B}^{*}}{{\left| X \right|}^{2\left( 1-v \right)}}B \right)}^{r}} \right\|,\quad\text{ }r\ge 1,~0<v<1.
\end{equation}
Some interesting numerical radius inequalities improving inequalities \eqref{14} have been obtained by several mathematicians (see \cite{9,10}, and references therein). For a comprehensive overview of the connections among these and other known inequalities in the literature, we refer to \cite{bookd}.

The purpose of this work is to establish some new inequalities for the numerical radius of bounded linear operators in Hilbert spaces. We provide a new estimate for the sum of two operators. After that, we generalize and improve the inequality \eqref{20}. An improvement of inequality $w\left( A \right)\le \left\| A \right\|$ is also given in the end of Section 2. Section 3 devoted to studying numerical radius inequalities involving $f$-connection of operators.

\section{Inequalities for sums and products of operators}
We start this section by an operator norm inequality related to \eqref{3}. In fact we give another upper bound for ${{\left\| A+B \right\|}^{2}}$.
\begin{theorem}\label{8}
	Let $A,B\in \mathbb{B}\left( \mathcal{H} \right)$, then
	\begin{equation}\label{26}
{{\left\| A+B \right\|}^{2}}\le \frac{1}{2}\left[ \left\| {{\left| {{A}^{*}} \right|}^{2}}+{{\left| {{B}^{*}} \right|}^{2}} \right\|+\left\| {{\left| {{A}^{*}} \right|}^{2}}-{{\left| {{B}^{*}} \right|}^{2}} \right\| \right]+w\left( B{{A}^{*}} \right)+2\left\| A \right\|\left\| B \right\|.	
	\end{equation}
\end{theorem}
\begin{proof}
	We use the following inequality which is shown in the proof of Theorem 3 in \cite{dragomir}:
	\[{{\left| \left\langle z,x \right\rangle  \right|}^{2}}+{{\left| \left\langle z,y \right\rangle  \right|}^{2}}\le {{\left\| z \right\|}^{2}}\max \left( {{\left\| x \right\|}^{2}},{{\left\| y \right\|}^{2}} \right)+\left| \left\langle x,y \right\rangle  \right|\]
	where $x,y,z\in \mathcal{H}$. Taking $x={{A}^{*}}y$,  $y={{B}^{*}}y$, and $z=x$ with $\left\| x \right\|=\left\| y \right\|=1$, we get  
	\[{{\left| \left\langle x,{{A}^{*}}y \right\rangle  \right|}^{2}}+{{\left| \left\langle x,{{B}^{*}}y \right\rangle  \right|}^{2}}\le \max \left( {{\left\| {{A}^{*}}y \right\|}^{2}},{{\left\| {{B}^{*}}y \right\|}^{2}} \right)+\left| \left\langle {{A}^{*}}y,{{B}^{*}}y \right\rangle  \right|.\]
	The above inequality is equivalent to
	\[{{\left| \left\langle Ax,y \right\rangle  \right|}^{2}}+{{\left| \left\langle Bx,y \right\rangle  \right|}^{2}}\le \frac{1}{2}\left[ \left\langle A{{A}^{*}}+B{{B}^{*}}y,y \right\rangle +\left| \left\langle A{{A}^{*}}-B{{B}^{*}}y,y \right\rangle  \right| \right]+\left| \left\langle B{{A}^{*}}y,y \right\rangle  \right|\]
	thanks to $\max \left\{ a,b \right\}=\frac{1}{2}\left( a+b+\left| a-b \right| \right)\left( a,b\in \mathbb{R} \right)$.
	
	Now, it follows from the tringle inequality that
	\[\begin{aligned}
	& {{\left| \left\langle A+Bx,y \right\rangle  \right|}^{2}} \\ 
	& \le {{\left| \left\langle Ax,y \right\rangle  \right|}^{2}}+{{\left| \left\langle Bx,y \right\rangle  \right|}^{2}}+2\left| \left\langle Ax,y \right\rangle  \right|\left| \left\langle Bx,y \right\rangle  \right| \\ 
	& \le \frac{1}{2}\left[ \left\langle A{{A}^{*}}+B{{B}^{*}}y,y \right\rangle +\left| \left\langle A{{A}^{*}}-B{{B}^{*}}y,y \right\rangle  \right| \right]+\left| \left\langle B{{A}^{*}}y,y \right\rangle  \right|+2\left| \left\langle Ax,y \right\rangle  \right|\left| \left\langle Bx,y \right\rangle  \right|. 
	\end{aligned}\]
	By taking the supremum over $x,y\in \mathcal{H}$ with $\left\| x \right\|=\left\| y \right\|=1$, we deduce the desired result.
\end{proof}

The following examples show that there is no ordering between our inequality \eqref{26} and Kittaneh inequality \eqref{3} in general.
\begin{example}
Let $A=\left( \begin{matrix}
1 & 0  \\
-3 & 1  \\
\end{matrix} \right)$
, $B=\left( \begin{matrix}
-1 & 2  \\
0 & 1  \\
\end{matrix} \right)$. After brief computation, 
	\[{{\left\| A+B \right\|}^{2}}\approx 14.52,\]
\[\frac{1}{2}\left[ \left\| {{\left| {{A}^{*}} \right|}^{2}}+{{\left| {{B}^{*}} \right|}^{2}} \right\|+\left\| {{\left| {{A}^{*}} \right|}^{2}}-{{\left| {{B}^{*}} \right|}^{2}} \right\| \right]+w\left( B{{A}^{*}} \right)+2\left\| A \right\|\left\| B \right\|\approx 29.58,\] 
and
\[\left\| {{\left| A \right|}^{2}}+{{\left| B \right|}^{2}} \right\|+\left\| {{\left| {{A}^{*}} \right|}^{2}}+{{\left| {{B}^{*}} \right|}^{2}} \right\|\approx 25.28.\] 
Thus,
\[\begin{aligned}
 {{\left\| A+B \right\|}^{2}}&\lneqq	\left\| {{\left| A \right|}^{2}}+{{\left| B \right|}^{2}} \right\|+\left\| {{\left| {{A}^{*}} \right|}^{2}}+{{\left| {{B}^{*}} \right|}^{2}} \right\| \\ 
& \lneqq	\frac{1}{2}\left[ \left\| {{\left| {{A}^{*}} \right|}^{2}}+{{\left| {{B}^{*}} \right|}^{2}} \right\|+\left\| {{\left| {{A}^{*}} \right|}^{2}}-{{\left| {{B}^{*}} \right|}^{2}} \right\| \right]+w\left( B{{A}^{*}} \right)+2\left\| A \right\|\left\| B \right\|.  
\end{aligned}\]
\end{example}

\begin{example}
Let $A=\left( \begin{matrix}
2 & 0  \\
3 & 1  \\
\end{matrix} \right)$
, $B=\left( \begin{matrix}
0 & 1  \\
0 & 1  \\
\end{matrix} \right)$. A simple computation shows that 
	\[{{\left\| A+B \right\|}^{2}}\approx 17.94,\]
\[\frac{1}{2}\left[ \left\| {{\left| {{A}^{*}} \right|}^{2}}+{{\left| {{B}^{*}} \right|}^{2}} \right\|+\left\| {{\left| {{A}^{*}} \right|}^{2}}-{{\left| {{B}^{*}} \right|}^{2}} \right\| \right]+w\left( B{{A}^{*}} \right)+2\left\| A \right\|\left\| B \right\|\approx 25.4,\] 
and
\[\left\| {{\left| A \right|}^{2}}+{{\left| B \right|}^{2}} \right\|+\left\| {{\left| {{A}^{*}} \right|}^{2}}+{{\left| {{B}^{*}} \right|}^{2}} \right\|\approx 29.44.\] 
Thus,
\[\begin{aligned}
 {{\left\| A+B \right\|}^{2}}&\lneqq	\frac{1}{2}\left[ \left\| {{\left| {{A}^{*}} \right|}^{2}}+{{\left| {{B}^{*}} \right|}^{2}} \right\|+\left\| {{\left| {{A}^{*}} \right|}^{2}}-{{\left| {{B}^{*}} \right|}^{2}} \right\| \right]+w\left( B{{A}^{*}} \right)+2\left\| A \right\|\left\| B \right\| \\ 
& \lneqq	\left\| {{\left| A \right|}^{2}}+{{\left| B \right|}^{2}} \right\|+\left\| {{\left| {{A}^{*}} \right|}^{2}}+{{\left| {{B}^{*}} \right|}^{2}} \right\|.  
\end{aligned}\]
\end{example}

\begin{remark}
It follows from Theorem \ref{8} that
\[{{\left\| A+B \right\|}^{2}}\le \frac{1}{2}\left[ \left\| {{\left| A \right|}^{2}}+{{\left| B \right|}^{2}} \right\|+\left\| {{\left| A \right|}^{2}}-{{\left| B \right|}^{2}} \right\| \right]+w\left( B{{A}^{*}} \right)+2\left\| A \right\|\left\| B \right\|,\]
whenever $A$ and $B$ are two normal operators.
\end{remark}

Letting $x = y$ in the proof of Theorem \ref{8}, we find that:
\begin{corollary}
	Let $A,B\in \mathbb{B}\left( \mathcal{H} \right)$, then
	\[{{w}^{2}}\left( A+B \right)\le \frac{1}{2}\left[ \left\| {{\left| {{A}^{*}} \right|}^{2}}+{{\left| {{B}^{*}} \right|}^{2}} \right\|+\left\| {{\left| {{A}^{*}} \right|}^{2}}-{{\left| {{B}^{*}} \right|}^{2}} \right\| \right]+w\left( B{{A}^{*}} \right)+2w\left( A \right)w\left( B \right).\]
\end{corollary}

The following lemmas are useful for generalizing and improving inequality \eqref{20}. The first lemma is known as the generalized mixed Schwarz inequality (see, e.g.,  {\cite[Theorem 1]{new}}).
\begin{lemma}\label{1}
	Let $A\in \mathbb{B}\left( \mathcal{H} \right)$ and $x,y\in \mathcal{H}$ be any vectors.
	If $f,g$ are non-negative continuous functions on $\left[ 0,\infty  \right)$ satisfying $f\left( t \right)g\left( t \right)=t,\left( t\ge 0 \right)$, then 
	$${{\left| \left\langle Ax,y \right\rangle  \right|}}\le \left\| f\left( \left| A \right| \right)x \right\|\left\| g\left( \left| {{A}^{*}} \right| \right)y \right\|.$$
\end{lemma}

The second lemma is well known in the literature as the Mond--Pe\v cari\'c inequality \cite{mond}.
\begin{lemma}\label{21}
If $f$ is a convex function on a real interval $J$ containing the spectrum of the self-adjoint operator $A$, then for any unit vector $x\in \mathcal{H}$,
\begin{equation}\label{conv_inner_ineq}
f\left(\left<Ax,x\right>\right)\leq \left<f(A)x,x\right>
\end{equation}
and the reverse inequality holds if $f$ is concave. 
\end{lemma}
The third lemma is a direct consequence of \cite[Theorem 2.3]{aujla}.
\begin{lemma}\label{18}
	Let $f$ be a non-negative non-decreasing convex function on $\left[ 0,\infty  \right)$ and let $A,B\in \mathbb{B}\left( \mathcal{H} \right)$ be positive operators. Then for any $0<v<1$,
\[\left\| f(\left( 1-v \right)A+vB) \right\|\le \left\| \left( 1-v \right)f\left( A \right)+vf\left( B \right) \right\|.\]
\end{lemma}

The above three lemmas admit the following more general result.
\begin{proposition}\label{30}
	Let $A,B,X\in \mathbb{B}\left( \mathcal{H} \right)$, and let $f$ and $g$ be non-negative functions on $\left[ 0,\infty  \right)$ which are continuous and satisfy the relation $f\left( t \right)g\left( t \right)=t$ for all $t\in \left[ 0,\infty  \right)$. If $h$ is a non-negative increasing convex function on $\left[ 0,\infty  \right)$, then for any $0<v<1$
	\begin{equation}\label{13}
h\left( {{w}^{2}}\left( {{A}^{*}}XB \right) \right)\le \left\| \left( 1-v \right)h\left( {{\left( {{B}^{*}}{{f}^{2}}\left( \left| X \right| \right)B \right)}^{\frac{1}{1-v}}} \right)+vh\left( {{\left( {{A}^{*}}{{g}^{2}}\left( \left| {{X}^{*}} \right| \right)A \right)}^{\frac{1}{v}}} \right) \right\|.	
	\end{equation}
	In particular, 
	\begin{equation}\label{24}
{{w}^{2r}}\left( {{A}^{*}}XB \right)\le \frac{1}{2}\left\| {{\left( {{B}^{*}}{{f}^{2}}\left( \left| X \right| \right)B \right)}^{2r}}+{{\left( {{A}^{*}}{{g}^{2}}\left( \left| {{X}^{*}} \right| \right)A \right)}^{2r}} \right\|	
	\end{equation}
	for all $r\ge 1$.
\end{proposition}
\begin{proof}
For any unit vector $x\in \mathcal{H}$, we have
\begin{align}
\nonumber {{\left| \left\langle {{A}^{*}}XBx,x \right\rangle  \right|}^{2}}&={{\left| \left\langle XBx,Ax \right\rangle  \right|}^{2}} \\ 
& \le \left\langle {{B}^{*}}{{f}^{2}}\left( \left| X \right| \right)Bx,x \right\rangle \left\langle {{A}^{*}}{{g}^{2}}\left( \left| {{X}^{*}} \right| \right)Ax,x \right\rangle  \label{9}\\ 
& =\left\langle {{\left( {{\left( {{B}^{*}}{{f}^{2}}\left( \left| X \right| \right)B \right)}^{\frac{1}{1-v}}} \right)}^{1-v}}x,x \right\rangle \left\langle {{\left( {{\left( {{A}^{*}}{{g}^{2}}\left( \left| {{X}^{*}} \right| \right)A \right)}^{\frac{1}{v}}} \right)}^{v}}x,x \right\rangle \nonumber \\ 
& \le {{\left\langle {{\left( {{B}^{*}}{{f}^{2}}\left( \left| X \right| \right)B \right)}^{\frac{1}{1-v}}}x,x \right\rangle }^{1-v}}{{\left\langle {{\left( {{A}^{*}}{{g}^{2}}\left( \left| {{X}^{*}} \right| \right)A \right)}^{\frac{1}{v}}}x,x \right\rangle }^{v}} \label{10} \\ 
& \le \left( 1-v \right)\left\langle {{\left( {{B}^{*}}{{f}^{2}}\left( \left| X \right| \right)B \right)}^{\frac{1}{1-v}}}x,x \right\rangle +v\left\langle {{\left( {{A}^{*}}{{g}^{2}}\left( \left| {{X}^{*}} \right| \right)A \right)}^{\frac{1}{v}}}x,x \right\rangle  \label{16}\\ 
& =\left\langle \left( 1-v \right){{\left( {{B}^{*}}{{f}^{2}}\left( \left| X \right| \right)B \right)}^{\frac{1}{1-v}}}+v{{\left( {{A}^{*}}{{g}^{2}}\left( \left| {{X}^{*}} \right| \right)A \right)}^{\frac{1}{v}}}x,x \right\rangle \nonumber  
\end{align}
where \eqref{9} follows from Lemma \ref{1}, \eqref{10} follows from Mond--Pe\v cari\'c inequality for concave function $f\left( t \right)={{t}^{v}}\left( 0<v<1 \right)$, and the weighted arithmetic-geometric mean inequality implies \eqref{16}.

Taking the supremum over $x\in \mathcal{H}$ with $\left\| x \right\|=1$, we infer that
\[{{w}^{2}}\left( {{A}^{*}}XB \right)\le \left\| \left( 1-v \right){{\left( {{B}^{*}}{{f}^{2}}\left( \left| X \right| \right)B \right)}^{\frac{1}{1-v}}}+v{{\left( {{A}^{*}}{{g}^{2}}\left( \left| {{X}^{*}} \right| \right)A \right)}^{\frac{1}{v}}} \right\|.\]
On account of assumptions on $h$, we can write
\begin{align}
\nonumber h\left( {{w}^{2}}\left( {{A}^{*}}XB \right) \right)&\le h\left( \left\| \left( 1-v \right){{\left( {{B}^{*}}{{f}^{2}}\left( \left| X \right| \right)B \right)}^{\frac{1}{1-v}}}+v{{\left( {{A}^{*}}{{g}^{2}}\left( \left| {{X}^{*}} \right| \right)A \right)}^{\frac{1}{v}}} \right\| \right) \\ 
& =\left\| h\left( \left( 1-v \right){{\left( {{B}^{*}}{{f}^{2}}\left( \left| X \right| \right)B \right)}^{\frac{1}{1-v}}}+v{{\left( {{A}^{*}}{{g}^{2}}\left( \left| {{X}^{*}} \right| \right)A \right)}^{\frac{1}{v}}} \right) \right\| \nonumber\\ 
& \le \left\| \left( 1-v \right)h\left( {{\left( {{B}^{*}}{{f}^{2}}\left( \left| X \right| \right)B \right)}^{\frac{1}{1-v}}} \right)+vh\left( {{\left( {{A}^{*}}{{g}^{2}}\left( \left| {{X}^{*}} \right| \right)A \right)}^{\frac{1}{v}}} \right) \right\|  \label{19}
\end{align}
where \eqref{19} follows from Lemma \ref{18}.

The inequality \eqref{24} follows directly from \eqref{13} by taking $h\left( t \right)={{t}^{r}}\left( r\ge 1 \right)$ and $v=\frac{1}{2}$.
\end{proof}

Our aim in the next result is to improve \eqref{20} under some mild conditions. To do this end, we need the following refinement of arithmetic-geometric mean inequality \cite{moradi1,moradi2}.
\begin{lemma}\label{4}
	Suppose that $a,b>0$ and positive real numbers $m$, $M$ satisfy $\min \left\{ a,b \right\}\le m<M\le \max \left\{ a,b \right\}$. 	Then
	\[\frac{M+m}{2\sqrt{Mm}}\sqrt{ab}\le \frac{a+b}{2}.\]
\end{lemma}
\begin{proof}
Consider $f\left( x \right)=\frac{2\sqrt{x}}{1+x}$ on $\left( 1\le  \right)\frac{M}{m}\le x$. Since $f'\left( x \right)=\frac{1-x}{\sqrt{x}{{\left( x+1 \right)}^{2}}}\le 0,\left( x\ge 1 \right)$ we get $f\left( x \right)\le f\left( \frac{M}{m} \right)$, which implies the result by a simple calculation.  
\end{proof}
\begin{theorem}\label{23}
	Let $A,B,X\in \mathbb{B}\left( \mathcal{H} \right)$, $f$ and $g$ be non-negative functions on $\left[ 0,\infty  \right)$ which are continuous and satisfy the relation $f\left( t \right)g\left( t \right)=t$ for all $t\in \left[ 0,\infty  \right)$, and let $h$ be a non-negative increasing convex function on $\left[ 0,\infty  \right)$. If
	\begin{equation*}
	0<{{B}^{*}}{{f}^{2}}\left( \left| X \right| \right)B\le m<M\le {{A}^{*}}{{g}^{2}}\left( \left| {{X}^{*}} \right| \right)A
	\end{equation*}
	or
	\[0< {{A}^{*}}{{g}^{2}}\left( \left| {{X}^{*}} \right| \right)A\le m<M\le {{B}^{*}}{{f}^{2}}\left( \left| X \right| \right)B,\]
	then   
	\begin{equation}\label{5}
h\left( w\left( {{A}^{*}}XB \right) \right)\le \frac{\sqrt{Mm}}{M+m}\left\| h\left( {{B}^{*}}{{f}^{2}}\left( \left| X \right| \right)B \right)+h\left( {{A}^{*}}{{g}^{2}}\left( \left| {{X}^{*}} \right| \right)A \right) \right\|.
	\end{equation} 
\end{theorem}

\begin{proof}
	It follows from Lemma \ref{1} that
	\begin{equation}\label{6}
\left| \left\langle {{A}^{*}}XBx,x \right\rangle  \right|\le \sqrt{\left\langle {{B}^{*}}{{f}^{2}}\left( \left| X \right| \right)Bx,x \right\rangle \left\langle {{A}^{*}}{{g}^{2}}\left( \left| {{X}^{*}} \right| \right)Ax,x \right\rangle }.
	\end{equation}
Lemma \ref{4} ensures that
	\begin{equation}\label{7}
	\begin{aligned}
	& \sqrt{\left\langle {{B}^{*}}{{f}^{2}}\left( \left| X \right| \right)Bx,x \right\rangle \left\langle {{A}^{*}}{{g}^{2}}\left( \left| {{X}^{*}} \right| \right)Ax,x \right\rangle } \\ 
	& \le \frac{\sqrt{Mm}}{M+m}\left( \left\langle {{B}^{*}}{{f}^{2}}\left( \left| X \right| \right)Bx,x \right\rangle +\left\langle {{A}^{*}}{{g}^{2}}\left( \left| {{X}^{*}} \right| \right)Ax,x \right\rangle  \right) \\ 
	& =\frac{\sqrt{Mm}}{M+m}\left\langle {{B}^{*}}{{f}^{2}}\left( \left| X \right| \right)B+{{A}^{*}}{{g}^{2}}\left( \left| {{X}^{*}} \right| \right)Ax,x \right\rangle.
	\end{aligned}
	\end{equation}
	Combining \eqref{6} and \eqref{7}, we get
	\[\left| \left\langle {{A}^{*}}XBx,x \right\rangle  \right|\le \frac{\sqrt{Mm}}{M+m}\left\langle {{B}^{*}}{{f}^{2}}\left( \left| X \right| \right)B+{{A}^{*}}{{g}^{2}}\left( \left| {{X}^{*}} \right| \right)Ax,x \right\rangle .\]
	Taking the supremum over $x\in \mathcal{H}$ with $\left\| x \right\|=1$, we infer that
	\[w\left( {{A}^{*}}XB \right)\le \frac{\sqrt{Mm}}{M+m}\left\| {{B}^{*}}{{f}^{2}}\left( \left| X \right| \right)B+{{A}^{*}}{{g}^{2}}\left( \left| {{X}^{*}} \right| \right)A \right\|.\]
	Now, since $h$ is a non-negative increasing convex function, we have
	\begin{align}
	h\left( w\left( {{A}^{*}}XB \right) \right)&\le h\left( \frac{2\sqrt{Mm}}{M+m}\left\| \frac{{{B}^{*}}{{f}^{2}}\left( \left| X \right| \right)B+{{A}^{*}}{{g}^{2}}\left( \left| {{X}^{*}} \right| \right)A}{2} \right\| \right) \nonumber\\ 
	& \le \frac{2\sqrt{Mm}}{M+m}h\left( \left\| \frac{{{B}^{*}}{{f}^{2}}\left( \left| X \right| \right)B+{{A}^{*}}{{g}^{2}}\left( \left| {{X}^{*}} \right| \right)A}{2} \right\| \right) \label{2}\\ 
	& = \frac{2\sqrt{Mm}}{M+m}\left\| h\left( \frac{{{B}^{*}}{{f}^{2}}\left( \left| X \right| \right)B+{{A}^{*}}{{g}^{2}}\left( \left| {{X}^{*}} \right| \right)A}{2} \right) \right\| \nonumber\\ 
	& \le \frac{\sqrt{Mm}}{M+m}\left\| h\left( {{B}^{*}}{{f}^{2}}\left( \left| X \right| \right)B \right)+h\left( {{A}^{*}}{{g}^{2}}\left( \left| {{X}^{*}} \right| \right)A \right) \right\| \label{17} 
	\end{align}
	where the inequality \eqref{2} follows from the fact if $f$ is non-negative convex function and $\alpha \le 1$, then $f\left( \alpha t \right)\le \alpha f\left( t \right)$ (of course, $\frac{2\sqrt{Mm}}{M+m}\le 1$),  and the inequality \eqref{17} is due to Lemma \ref{18}.
\end{proof}

\begin{remark}\label{remark2.2}
Following \eqref{5} we list here some particular inequalities of interest.
	\begin{itemize}
		\item If $r\ge 1$ and $0\le v\le 1$, then
		\[{{w}^{r}}\left( {{A}^{*}}XB \right)\le \frac{\sqrt{Mm}}{M+m}\left\| {{\left( {{B}^{*}}{{\left| X \right|}^{2\left( 1-v \right)}}B \right)}^{r}}+{{\left( {{A}^{*}}{{\left| {{X}^{*}} \right|}^{2v}}A \right)}^{r}} \right\|\]
		whenever $0<{{B}^{*}}{{\left| X \right|}^{2\left( 1-v \right)}}B\le m<M\le {{A}^{*}}{{\left| {{X}^{*}} \right|}^{2v}}A$ or $0< {{A}^{*}}{{\left| {{X}^{*}} \right|}^{2v}}A\le m<M\le {{B}^{*}}{{\left| X \right|}^{2\left( 1-v \right)}}B$.

		The above inequality improves \eqref{20}.  
		\item If $r\ge 1$ and $0\le v\le 1$, then
		\[{{w}^{r}}\left( X \right)\le \frac{\sqrt{Mm}}{M+m}\left\| {{\left| X \right|}^{2r\left( 1-v \right)}}+{{\left| {{X}^{*}} \right|}^{2rv}} \right\|\]
		whenever $0<{{\left| X \right|}^{2\left( 1-v \right)}}\le m<M\le {{\left| {{X}^{*}} \right|}^{2v}}$ or $0< {{\left| {{X}^{*}} \right|}^{2v}}\le m<M\le {{\left| X \right|}^{2\left( 1-v \right)}}$.
		
		The above inequality improves \eqref{28}.
		\item If $r\ge 1$, then
		\[{{w}^{r}}\left( {{A}^{*}}B \right)\le \frac{\sqrt{Mm}}{M+m}\left\| {{\left| B \right|}^{2r}}+{{\left| A \right|}^{2r}} \right\|.\]
		whenever $0< {{\left| B \right|}^{2}}\le m<M\le {{\left| A \right|}^{2}}$ or $0<{{\left| A \right|}^{2}}\le m<M\le {{\left| B \right|}^{2}}$.
		
		The above inequality improves \eqref{29}.
	\end{itemize}
\end{remark}

We can show a similar improvement with different condition for $A^*g^2\left(\vert X \vert \right) A$ and $B^*f^2\left(\vert X \vert \right) B$.
Recall that the weighted operator arithmetic mean ${{\nabla }_{v}}$ and geometric mean ${{\sharp}_{v}}$, for $0<v<1$, positive invertible operator $A$, and positive operator $B$, are defined as follows:
\[A{{\nabla }_{v}}B=\left( 1-v \right)A+vB\quad\text{ and }\quad A{{\sharp}_{v}}B={{A}^{\frac{1}{2}}}\left( {{A}^{-\frac{1}{2}}}B{{A}^{-\frac{1}{2}}} \right)^v{{A}^{\frac{1}{2}}}.\]
If $v=\frac{1}{2}$, we denote the arithmetic and geometric means, respectively, by $\nabla $ and $\sharp$.
\begin{theorem}
Let $A,B,X\in \mathbb{B}\left( \mathcal{H} \right)$, $f$ and $g$ be non-negative functions on $\left[ 0,\infty  \right)$ which are continuous and satisfy the relation $f\left( t \right)g\left( t \right)=t$ for all $t\in \left[ 0,\infty  \right)$, and let $h$ be a non-negative increasing convex function on $\left[ 0,\infty  \right)$. If for given $m',M' >0$,
$$0<m' \leq B^*f^2\left(\vert X \vert \right) B \leq A^*g^2\left(\vert X \vert \right) A \leq M'$$
or
$$0<m' \leq A^*g^2\left(\vert X \vert \right) A  \leq  B^*f^2\left(\vert X \vert \right) B\leq M',$$
then
\[\,h\left( {\omega \left( {{A^*}XB} \right)} \right) \le \frac{1}{{2\gamma }}\left\| {h\left( {{B^*}{f^2}\left( {\left| X \right|} \right)B} \right) + h\left( {{A^*}{g^2}\left( {\left| {{X^*}} \right|} \right)A} \right)} \right\|,\]
where $\gamma : = {\left( {1 - \dfrac{1}{8}{{\left( {1 - \dfrac{1}{{h'}}} \right)}^2}} \right)^{ - 1}} \ge 1$ with $h' = \frac{M'}{m'}$.
\end{theorem}

\begin{proof}
From \cite[Corollary 3.15]{Furuichi2019}, we have
\[{\exp _r}\left( {\frac{{v\left( {1 - v} \right)}}{2}{{\left( {1 - \frac{1}{{h'}}} \right)}^2}} \right)A{\sharp _v}B \le A{\nabla _v}B\]
for $A,B > 0$ with $m',M' > 0$ satisfying $0 < m'\le A \le B \le M'$ or $0 < m'\le B \le A \le M'$, where ${\exp _r}\left( x \right): = {\left( {1 + rx} \right)^{1/r}},\,{\rm{if}}\,\,1 + rx > 0$, and it is undefined otherwise.
Since $\exp_r(x)$ is decreasing in $r \in [-1,0)$, the above inequality gives a tight lower bound when $r=-1.$
After all, we have the scalar inequality:
\[\gamma \sqrt {ab}  \le \frac{{a + b}}{2}\]
for $a,b>0$ and $m',M'>0$ such that 
$0<m' \leq \min\{a,b\} \leq \max \{a,b\} \leq M'$.
Applying this inequality with a similar argument as in Theorem \ref{23}, we obtain the desired result.
\end{proof}

We also obtain the similar remarks with Remark \ref{remark2.2}, we omit them.

As we have seen, Lemma \ref{18} played an essential role in Proposition \ref{30} and Theorem \ref{23}.  In the following, we aim to improve Lemma \ref{18}.
\begin{proposition}\label{005}
	Let the assumptions of Lemma \ref{18} hold. Then
	\begin{equation}\label{31}
\left\| f\left( \left( 1-v \right)A+vB \right) \right\|\le \left\| \left( 1-v \right)f\left( A \right)+vf\left( B \right) \right\|-r\mu \left( f \right)	
	\end{equation}
	where $r=\min \left\{ v,1-v \right\}$, and
	\begin{equation}\label{002}
	\mu \left( f \right)=\underset{\left\| x \right\|=1}{\mathop{\inf }}\,\left\{ f\left( \left\langle Ax,x \right\rangle  \right)+f\left( \left\langle Bx,x \right\rangle  \right)-2f\left( \left\langle \left( \frac{A+B}{2} \right)x,x \right\rangle  \right) \right\}.	
	\end{equation}
\end{proposition}
\begin{proof}
	We assume $0\le v\le \frac{1}{2}$. For each unit vector $x\in \mathcal{H}$,
	\begin{align}
	\nonumber f\left( \left\langle \left( \left( 1-v \right)A+vB \right)x,x \right\rangle  \right)+r\mu \left( f \right)&=f\left( \left( 1-v \right)\left\langle Ax,x \right\rangle +v\left\langle Bx,x \right\rangle  \right)+r\mu \left( f \right) \nonumber\\ 
	& =f\left( \left( 1-2v \right)\left\langle Ax,x \right\rangle +2v\left\langle \left( \frac{A+B}{2} \right)x,x \right\rangle  \right)+r\mu \left( f \right) \nonumber\\ 
	& \le \left( 1-2v \right)f\left( \left\langle Ax,x \right\rangle  \right)+2vf\left( \left\langle \left( \frac{A+B}{2} \right)x,x \right\rangle  \right)+r\mu \left( f \right) \label{001}\\ 
	& \le \left( 1-2v \right)f\left( \left\langle Ax,x \right\rangle  \right)+2vf\left( \left\langle \left( \frac{A+B}{2} \right)x,x \right\rangle  \right) \label{003}\\ 
	&\quad +r\left( f\left( \left\langle Ax,x \right\rangle  \right)+f\left( \left\langle Bx,x \right\rangle  \right)-2f\left( \left\langle \left( \frac{A+B}{2} \right)x,x \right\rangle  \right) \right) \nonumber\\ 
	& =\left( 1-v \right)f\left( \left\langle Ax,x \right\rangle  \right)+vf\left( \left\langle Bx,x \right\rangle  \right) \nonumber\\ 
	& \le \left( 1-v \right)\left\langle f\left( A \right)x,x \right\rangle +v\left\langle f\left( B \right)x,x \right\rangle \label{004} \\ 
	& =\left\langle \left( \left( 1-v \right)f\left( A \right)+vf\left( B \right) \right)x,x \right\rangle   \nonumber
	\end{align}
	where \eqref{001} follows from convexity of $f$, the relation \eqref{002} implies \eqref{003}, and \eqref{004} follows from Lemma \ref{21}.
	
If we apply similar arguments for $\frac{1}{2}\le v\le 1$, then we can write
	\[f\left( \left\langle \left( \left( 1-v \right)A+vB \right)x,x \right\rangle  \right)\le \left\langle \left( \left( 1-v \right)f\left( A \right)+vf\left( B \right) \right)x,x \right\rangle -r\mu \left( f \right).\]
We know that if $A\in \mathbb{B}\left( \mathcal{H} \right)$ is a positive operator, then $\left\| A \right\|={{\sup }_{\left\| x \right\|=1}}\left\langle Ax,x \right\rangle $. By using this, the continuity and the increase of $f$, we have
\[\begin{aligned}
 f\left( \left\| \left( 1-v \right)A+vB \right\| \right)&=f\left( \underset{\left\| x \right\|=1}{\mathop{\sup }}\,\left\langle \left( \left( 1-v \right)A+vB \right)x,x \right\rangle  \right) \\ 
& =\underset{\left\| x \right\|=1}{\mathop{\sup }}\,f\left( \left\langle \left( \left( 1-v \right)A+vB \right)x,x \right\rangle  \right) \\ 
& \le \underset{\left\| x \right\|=1}{\mathop{\sup }}\,\left( \left\langle \left( \left( 1-v \right)f\left( A \right)+vf\left( B \right) \right)x,x \right\rangle  \right)-r\mu \left( f \right) \\ 
& =\left\| \left( 1-v \right)f\left( A \right)+vf\left( B \right) \right\|-r\mu \left( f \right).  
\end{aligned}\]
On the other hand, if $X\in \mathbb{B}\left( \mathcal{H} \right)$, and if $f$ is a non-negative increasing function on $\left[ 0,\infty  \right)$, then $f\left( \left\| X \right\| \right)=\left\| f\left( \left| X \right| \right) \right\|$, so we get the desired result.	
	
\end{proof}

\begin{remark}
With inequality \eqref{31} in hand, we can improve Proposition \ref{30} and Theorem \ref{23}. For instance, under the assumptions of Proposition \ref{30}, we have
\[h\left( {{w}^{2}}\left( {{A}^{*}}XB \right) \right)\le \left\| \left( 1-v \right)h\left( {{\left( {{B}^{*}}{{f}^{2}}\left( \left| X \right| \right)B \right)}^{\frac{1}{1-v}}} \right)+vh\left( {{\left( {{A}^{*}}{{g}^{2}}\left( \left| {{X}^{*}} \right| \right)A \right)}^{\frac{1}{v}}} \right) \right\|-r\gamma \left( f \right)\]
where
\[\begin{aligned}
 \gamma \left( f \right)&=\underset{\left\| x \right\|=1}{\mathop{\inf }}\,\left\{ h\left( \left\langle {{\left( {{B}^{*}}{{f}^{2}}\left( \left| X \right| \right)B \right)}^{\frac{1}{1-v}}}x,x \right\rangle  \right)+h\left( \left\langle {{\left( {{A}^{*}}{{g}^{2}}\left( \left| {{X}^{*}} \right| \right)A \right)}^{\frac{1}{v}}}x,x \right\rangle  \right) \right. \\ 
&\quad \left. -2h\left( \left\langle \left( \frac{{{\left( {{B}^{*}}{{f}^{2}}\left( \left| X \right| \right)B \right)}^{\frac{1}{1-v}}}+{{\left( {{A}^{*}}{{g}^{2}}\left( \left| {{X}^{*}} \right| \right)A \right)}^{\frac{1}{v}}}}{2} \right)x,x \right\rangle  \right) \right\}.
\end{aligned}\]
\end{remark}

Now we present some inequalities for the numerical radius and operator norm, but under the effect of a superquadratic function. Recall that a function $f:\left[ 0,\infty  \right)\to \mathbb{R}$ is said to be superquadratic provided that for all $s\ge 0,$ there exists a constant ${{C}_{s}}\in \mathbb{R}$ such that
\begin{equation}\label{01}
f\left( \left| t-s \right| \right)+{{C}_{s}}\left( t-s \right)+f\left( s \right)\le f\left( t \right)
\end{equation}
for all $t\ge 0$.

The following useful lemma is well known  \cite[Lemma 2.1]{01}.
\begin{lemma}\label{12}
	Suppose that $f$ is superquadratic and non-negative. Then $f$ is convex and increasing. Also, if ${{C}_{s}}$ is as in \eqref{01}, then ${{C}_{s}}\ge 0$.
\end{lemma}

By adopting the above notions, we can refine the second inequality in \eqref{14}.
\begin{theorem}\label{11}
	Let $A\in \mathbb{B}\left( \mathcal{H} \right)$ and let $f$ be a non-negative superquadratic function. Then
	\begin{equation}\label{05}
	f\left( w\left( A \right) \right)\le \left\| f\left( \left| A \right| \right) \right\|-\underset{\left\| x \right\|=1}{\mathop{\inf }}\,{{\left\| f{{\left( \left| \left| A \right|-w\left( A \right) \right| \right)}^{\frac{1}{2}}}x \right\|}^{2}}.
	\end{equation}
\end{theorem}
\begin{proof}
	Letting $s=w\left( A \right)$ in the inequality \eqref{01}, we get
	\begin{equation}\label{02}
	f\left( \left| t-w\left( A \right) \right| \right)+{{C}_{w\left( A \right)}}\left( t-w\left( A \right) \right)+f\left( w\left( A \right) \right)\le f\left( t \right).
	\end{equation}
	By applying functional calculus for the operator $\left| A \right|$ in \eqref{02} we get
	\begin{equation}\label{needed_1}
	f\left( \left| \left| A \right|-w\left( A \right) \right| \right)+{{C}_{w\left( A \right)}}\left( \left| A \right|-w\left( A \right) \right)+f\left( w\left( A \right) \right)\le f\left( \left| A \right| \right).
	\end{equation}
	Consequently,
	\begin{equation}\label{03}
	{{\left\| f{{\left( \left| \left| A \right|-w\left( A \right) \right| \right)}^{\frac{1}{2}}}x \right\|}^{2}}+{{C}_{w\left( A \right)}}\left( \left\langle \left| A \right|x,x \right\rangle -w\left( A \right) \right)+f\left( w\left( A \right) \right)\le \left\langle f\left( \left| A \right| \right)x,x \right\rangle
	\end{equation}
	for any unit vector $x\in \mathcal{H}$. 
	
	Now, by taking supremum over $x\in \mathcal{H}$ with $\left\| x \right\|=1$ in \eqref{03}, and using the fact $w\left( \left| A \right| \right)=\left\| A \right\|\ge w\left( A \right)$, we deduce the desired inequality \eqref{05}.
\end{proof}

Applying Theorem \ref{11} to the superquadratic function $f\left( t \right)={{t}^{r}}\left( r\ge 2 \right)$, we reach the following corollary:
\begin{corollary}
	Let $A\in \mathbb{B}\left( \mathcal{H} \right)$. Then for any $r\ge 2$,
	\[{{w}^{r}}\left( A \right)\le {{\left\| A \right\|}^{r}}-\underset{\left\| x \right\|=1}{\mathop{\inf }}\,{{\left\| {{\left| \left| A \right|-w\left( A \right) \right|}^{\frac{r}{2}}}x \right\|}^{2}}.\]
	In particular,
	\[w\left( A \right)\le \sqrt{{{\left\| A \right\|}^{2}}-\underset{\left\| x \right\|=1}{\mathop{\inf }}\,{{\left\| \left| \left| A \right|-w\left( A \right) \right|x \right\|}^{2}}}\le \left\| A \right\|.\]
\end{corollary}

\section{An inequality related to  f--connection of operators}
In the forthcoming, we aim to extend the main result of \cite{hosseini}.

 In \cite[Theorem 2.3]{hosseini}, the author tried to prove the numerical radius version of operator arithmetic-geometric mean inequality
\[{{w}^{r}}\left( \left( A\sharp B \right)X \right)\le w\left( \frac{{{A}^{\frac{rp}{2}}}}{p}+\frac{{{\left( {{X}^{*}}BX \right)}^{\frac{rq}{2}}}}{q} \right)-\frac{1}{p}\underset{\left\| x \right\|=1}{\mathop{\inf }}\,\delta \left( x \right)\]
where $A,B,X\in \mathbb{B}\left( \mathcal{H} \right)$ such that $A,B$ are positive invertible operators, $p\ge q>1$, $\frac{1}{p}+\frac{1}{q}=1$, $r\ge \frac{2}{q}$, and $\delta \left( x \right)={{\left( {{\left\langle Ax,x \right\rangle }^{\frac{rp}{4}}}-{{\left\langle {{X}^{*}}BXx,x \right\rangle }^{\frac{rq}{4}}} \right)}^{2}}$.

Of course, $\frac{{{A}^{\frac{rp}{2}}}}{p}+\frac{{{\left( {{X}^{*}}BX \right)}^{\frac{rq}{2}}}}{q}$ is positive. On the other hand, it is well-known to all that if $X$ is positive operator then $w\left( X \right)=\left\| X \right\|$. On taking into account these considerations, it should be written to the following form:
 \[{{w}^{r}}\left( \left( A\sharp B \right)X \right)\le \left\| \frac{{{A}^{\frac{rp}{2}}}}{p}+\frac{{{\left( {{X}^{*}}BX \right)}^{\frac{rq}{2}}}}{q} \right\|-\frac{1}{p}\underset{\left\| x \right\|=1}{\mathop{\inf }}\,\delta \left( x \right).\]
Of course, the geometric mean (resp. arithmetic mean) of two positive operators is also a positive operator. So Corollary 2.6, Corollary 2.7, Remark 2.8, and Corollary 2.10 in \cite{hosseini} should be written in the following way, respectively,
	\[{{\left\| A\sharp B \right\|}^{r}}\le \left\| \frac{{{A}^{\frac{rp}{2}}}}{p}+\frac{{{B}^{\frac{rq}{2}}}}{q} \right\|-\frac{1}{p}\underset{\left\| x \right\|=1}{\mathop{\inf }}\,\left\{ {{\left( {{\left\langle Ax,x \right\rangle }^{\frac{rp}{4}}}-{{\left\langle Bx,x \right\rangle }^{\frac{rq}{4}}} \right)}^{2}} \right\},\] 
\[{{\left\| A\sharp B \right\|}^{2r}}\le \left\| \frac{{{A}^{rp}}}{p}+\frac{{{B}^{rq}}}{q} \right\|-\frac{1}{p}\underset{\left\| x \right\|=1}{\mathop{\inf }}\,\left\{ {{\left( {{\left\langle Ax,x \right\rangle }^{\frac{rp}{2}}}-{{\left\langle Bx,x \right\rangle }^{\frac{rq}{2}}} \right)}^{2}} \right\},\]
\[{{\left\| A\sharp B \right\|}^{2}}\le \left\| \frac{{{A}^{2}}+{{B}^{2}}}{2} \right\|-\frac{1}{2}\underset{\left\| x \right\|=1}{\mathop{\inf }}\,\left\{ {{\left\langle A-Bx,x \right\rangle }^{2}} \right\},\]
and
\[\sqrt{2}\left\| A\sharp B \right\|\le {{w}_{e}}\left( A,B \right)\le {{\left\| {{A}^{2}}+{{B}^{2}} \right\|}^{\frac{1}{2}}}.\]
Here ${{w}_{e}}\left( A,B \right)=\underset{\left\| x \right\|=1}{\mathop{\sup }}\,{{\left( {{\left| \left\langle Ax,x \right\rangle  \right|}^{2}}+{{\left| \left\langle Bx,x \right\rangle  \right|}^{2}} \right)}^{\frac{1}{2}}}$.

Let $f$ be a continuous function defined on the real interval $J$ containing the spectrum of ${{A}^{-\frac{1}{2}}}B{{A}^{-\frac{1}{2}}}$, where $B$ is a self-adjoint operator and $A$ is a positive invertible operator. Then by using the continuous functional calculus, we can define $f$-connection ${{\sigma }_{f}}$ as follows
\begin{equation}\label{25}
A{{\sigma }_{f}}B={{A}^{\frac{1}{2}}}f\left( {{A}^{-\frac{1}{2}}}B{{A}^{-\frac{1}{2}}} \right){{A}^{\frac{1}{2}}}.
\end{equation}
Note that for the functions $\left( 1-v \right)+vt$ and ${{t}^{v}}$, the  definition in \eqref{25} leads to the arithmetic and geometric operator means, respectively.

Now, we give our numerical radius inequality concerning $f$-connection of operators.
\begin{theorem}\label{22}
Let $A,B,X\in \mathbb{B}\left( \mathcal{H} \right)$ such that $A,B$ be two positive operators. Then
\begin{equation}\label{27}
w\left( \left( A{{\sigma }_{f}}B \right)X \right)\le \frac{1}{2}\left\| {{X}^{*}}{{A}^{\frac{1}{2}}}{{f}^{2}}\left( {{A}^{-\frac{1}{2}}}B{{A}^{-\frac{1}{2}}} \right){{A}^{\frac{1}{2}}}X+A \right\|.
\end{equation}
\end{theorem}
\begin{proof}
For any unit vector $x\in \mathcal{H}$, we have
\[\begin{aligned}
 \left| \left\langle \left( A{{\sigma }_{f}}B \right)Xx,x \right\rangle  \right|&=\left| \left\langle {{A}^{\frac{1}{2}}}f\left( {{A}^{-\frac{1}{2}}}B{{A}^{-\frac{1}{2}}} \right){{A}^{\frac{1}{2}}}Xx,x \right\rangle  \right| \\ 
& =\left| \left\langle f\left( {{A}^{-\frac{1}{2}}}B{{A}^{-\frac{1}{2}}} \right){{A}^{\frac{1}{2}}}Xx,{{A}^{\frac{1}{2}}}x \right\rangle  \right| \\ 
& \le \left\| f\left( {{A}^{-\frac{1}{2}}}B{{A}^{-\frac{1}{2}}} \right){{A}^{\frac{1}{2}}}Xx \right\|\left\| {{A}^{\frac{1}{2}}}x \right\| \\ 
& =\sqrt{\left\langle f\left( {{A}^{-\frac{1}{2}}}B{{A}^{-\frac{1}{2}}} \right){{A}^{\frac{1}{2}}}Xx,f\left( {{A}^{-\frac{1}{2}}}B{{A}^{-\frac{1}{2}}} \right){{A}^{\frac{1}{2}}}Xx \right\rangle \left\langle {{A}^{\frac{1}{2}}}x,{{A}^{\frac{1}{2}}}x \right\rangle } \\ 
& =\sqrt{\left\langle {{X}^{*}}{{A}^{\frac{1}{2}}}{{f}^{2}}\left( {{A}^{-\frac{1}{2}}}B{{A}^{-\frac{1}{2}}} \right){{A}^{\frac{1}{2}}}Xx,x \right\rangle \left\langle Ax,x \right\rangle } \\ 
& \le \frac{1}{2}\left\langle {{X}^{*}}{{A}^{\frac{1}{2}}}{{f}^{2}}\left( {{A}^{-\frac{1}{2}}}B{{A}^{-\frac{1}{2}}} \right){{A}^{\frac{1}{2}}}X+Ax,x \right\rangle.
\end{aligned}\]
Now, the result follows by taking the supremum over $x\in \mathcal{H}$ with $\left\| x \right\|=1$.
\end{proof}
By choosing $f\left( t \right)=\sqrt{t}$, in Theorem \ref{22} we reach the following result:
\begin{corollary}
Let $A,B,X\in \mathbb{B}\left( \mathcal{H} \right)$ such that $A,B$ be two positive operators. Then
\[w\left( \left( A\sharp B \right)X \right)\le \frac{1}{2}\left\| {{X}^{*}}BX+A \right\|.\]
\end{corollary}
\begin{remark}
The interested reader can construct refinements of inequality \eqref{27} using improvements of weighted arithmetic-geometric mean inequality. We leave the details of this idea to the interested reader, as it is just an application of our result.
\end{remark}

\section*{Acknowledgements}
The authors would like to thank the anonymous reviewer for his/her comments.

\vskip 0.6 true cm

\tiny$^1$Department of Mathematics, Hormoz Branch, Islamic Azad University, Hormoz Island, Iran.

{\it E-mail address:} saratafazoli3@gmail.com

\vskip 0.4 true cm

\tiny$^2$Department of Mathematics, Payame Noor University (PNU), P.O. Box 19395-4697, Tehran, Iran.

{\it E-mail address:} hrmoradi@mshdiau.ac.ir

\vskip 0.4 true cm

\tiny$^3$Department of Information Science, College of Humanities and Sciences, Nihon University, 3-25-40, Sakurajyousui,
Setagaya-ku, Tokyo, 156-8550, Japan.

{\it E-mail address:} furuichi@chs.nihon-u.ac.jp

\vskip 0.4 true cm

\tiny$^4$Department of Mathematics, Manipal Institute of Technology, Manipal Academy of Higher Education, Deemed to be University, Manipal-576104, Karnataka, India.

{\it E-mail address:} pk.harikrishnan@manipal.edu
\end{document}